\documentclass[a4paper]{amsart}

\usepackage{amsfonts,amssymb,amsthm}
\usepackage[totalwidth=17cm,totalheight=22cm]{geometry}
\usepackage[T1]{fontenc}

\newtheorem{prop}{Proposition}[section]
\newtheorem{df}[prop]{Definition}
\newtheorem{lemma}[prop]{Lemma}
\newtheorem{thm}[prop]{Theorem}
\newtheorem{cor}[prop]{Corollary}

\newcommand{\II}{I\hspace{-0.1cm}I}
\newcommand{\III}{I\hspace{-0.1cm}I\hspace{-0.1cm}I}
\newcommand{\dr}{\partial}

\newcommand{\be}{\begin{eqnarray}}
\newcommand{\ee}{\end{eqnarray}}
\newcommand{\C}{{\mathbb C}}
\newcommand{\R}{{\mathbb R}}
\newcommand{\Z}{{\mathbb Z}}
\newcommand{\CP}{\mathcal{CP}}
\newcommand{\ML}{\mathcal{ML}}

\newcommand{\cG}{{\mathcal G}}
\newcommand{\cH}{{\mathcal H}}
\newcommand{\cS}{{\mathcal S}}
\newcommand{\cT}{{\mathcal T}}

\newcommand{\ld}{\dot{l}}
\newcommand{\md}{\dot{m}}
\newcommand{\ud}{\dot{u}}
\newcommand{\lb}{\overline{l}}
\newcommand{\Vb}{\overline{V}}

\begin{document}

\title{A symplectic map between hyperbolic and complex Teichm\"uller theory}
\date{v2: January 2009}
\author{Kirill Krasnov
~and Jean-Marc Schlenker
}
\address{K.K.: School of Mathematical Sciences, University of
    Nottingham, Nottingham, NG7 2RD, UK} 
\address{J.-M. S.:
Institut de Math{\'e}matiques de Toulouse, UMR CNRS 5219,
Universit{\'e} Toulouse III,
31062 Toulouse cedex 9, France.
\texttt{http://www.math.univ-toulouse.fr/\~{ }schlenker}.
}
\thanks{J.-M.S. was partially supported by the ANR program Repsurf : ANR-06-BLAN-0311.}

\begin{abstract} 
Let $S$ be a closed, orientable surface of genus at least $2$. The space 
$\cT_H\times \ML$, where $\cT_H$ is the ``hyperbolic'' Teichm\"uller space of $S$ and $\ML$ is
the space of measured geodesic laminations on $S$ is naturally a real symplectic manifold. 
The space $\CP$ of complex projective structures on $S$ is a complex symplectic manifold. 
A relation between these spaces is provided by the Thurston's grafting map $Gr$. 
We prove that this map, although not smooth, is symplectic. The proof uses a
variant of the renormalized volume defined for hyperbolic ends.
\end{abstract}

\maketitle

\section{Introduction and main results}

In all the paper $S$ is a closed, orientable surface of genus $g$ at
least $2$, $\cT$ is the Teichm\"uller space of $S$, $\CP$ the
space of complex projective structures on $S$, and $\ML$ be the
space of measured laminations on $S$. 

\subsection{The ``hyperbolic'' Teichm\"uller space}

There are several quite distinct ways to define the Teichm\"uller space of
$S$, e.g., as the space of complex structures on $S$, or as a space of (particular)
representations of $\pi_1(S)$ in $PSL(2,\R)$ (modulo conjugation). 
In this subsection we consider 
what can be called the ``hyperbolic'' Teichm\"uller space, defined as the space of 
hyperbolic metrics on $S$, considered up to isotopy.
In this guise it is sometimes called the Fricke space of $S$. Here we denote this
space by $\cT_H$ to remember its ``hyperbolic'' nature. This description emphasizes
geometric properties of $\cT$, while some other properties, notably the
complex structure on $\cT$, remain silent in it.

There is a natural identification between $\cT_H \times \ML$ and the 
cotangent bundle of $\cT_H$, which can be defined as follows.
Let $l\in \ML$ be a measured lamination on $S$. For each hyperbolic
metric $m\in \cT_H$ on $S$, let $L_m(l)$ be the geodesic length of $l$
for $m$ (as studied e.g. in \cite{kerckhoff,wolpert-length}). 
Thus $m\mapsto L_m(l)$ is a function on $\cT$, which is differentiable.
For $m_0\in \cT_H$, the differential of $m\mapsto L_m(l)$ at $m_0$ is
a vector in $T^*_{m_0}\cT_H$, which we call $\delta(m,l)$. This defines
a function $\delta:\cT_H\times \ML\rightarrow T^*\cT_H$, which is the identification
we wish to use here. It is proved in section 2 (see Lemma \ref{lm:delta})
that $\delta$ is indeed one-to-one (this fact should be quite obvious to
the specialists, a proof is included here for completeness). It was proved by 
Bonahon (see \cite{bonahon-toulouse}) that $\delta$ is tangentiable, that is, the image by 
$\delta$ of a tangent vector is well-defined. Moreover the tangent map is invertible
at each point.

Let $\omega_H$ denote the cotangent symplectic structure on $T^*\cT_H$. The map $\delta$
can be used to pull-back to $\omega_H$ to $\cT_H\times \ML$, therefore making 
$\cT_H\times \ML$ into a symplectic manifold. Somewhat abusing notations, we call
$\omega_H$ again the symplectic form on $\cT_H\times \ML$ obtained in this manner. A
more involved but explicit description on $\omega_H$ on $\cT_H\times \ML$ is recalled below
in subsection \ref{ssc:intersection}.

Note that the identification $\delta$ between $\cT_H\times \ML$ and $T^*\cT_H$ is 
{\it not} identical with the better known identification,
which goes through measured foliations and quadratic differentials, see e.g. \cite{FLP}. 

\subsection{The ``complex'' Teichm\"uller space and complex projective structures}

We now consider the ``complex'' Teichm\"uller space of $S$, denoted here by 
$\cT_C$, which is the space of complex structures on $S$. Of course there is a 
canonical identification between $\cT_H$ and $\cT_C$ --- there is a unique hyperbolic 
metric in each conformal class on $S$. However, as this map is not 
explicit, it appears helpful to keep in mind the distinction between
the two viewpoints. Note that the term ``complex'' could be used here in two different,
albeit related meanings. One is the above definition of $\cT_C$ as the space of complex 
structures on $S$. The other is related to the well-known deformation theory
of $\cT_C$ in terms of Beltrami differentials. It is in this point of view that
the complex structure on $\cT_C$ becomes manifest. So, it is useful to
keep in mind that the ``complex'' refers both to the complex structures on $S$
and on $\cT_C$ itself.

Now let $\CP$ denote the space of (equivalence classes of) $\C P^1$-structures 
(or complex projective structures) on $S$. 
Recall that a (complex) projective structure on $S$ is an maximal atlas of
charts from $S$ into $\C P^1$ such that all transition maps are 
M\"obius transformations. Such a structure naturally yields a holonomy
representation $hol: \pi_1(S) \to {\rm PSL}(2,\C)$, as well as an
$hol(\pi_1(S))$-equivariant developing map $dev:\tilde{S} \to \C P^1$.

There is a natural relation between complex projective structures
on $S$ and complex structures along with a holomorphic 
quadratic differential on $S$. Thus, let $\sigma$ be a complex projective structure on $S$, 
and let $\sigma_0$ be the ``Fuchsian'' $\C P^1$-structure on $S$ obtained by 
the Fuchsian uniformization of the complex structure underlying $\sigma$. Then 
the Schwarzian derivative of the complex map from $(S,\sigma_0)$ to $(S,\sigma)$ is 
a quadratic differential $q$ on $S$, holomorphic with respect to the
conformal structure $c$ of both $\sigma$ and $\sigma_0$, and the map sending 
$\sigma$ to $(c,q)$ is a homeomorphism, see e.g. \cite{dumas-survey,McMullen}. 

Recall also that the space of couples $(c,q)$ where $c$ is a complex structure on 
$S$ and $q$ is a quadratic holomorphic differential on $(S,c)$ is naturally
identified with the complexified cotangent bundle of $\cT_C$, see e.g. \cite{ahlfors}. 
Thus, $\CP$ is naturally a complex symplectic manifold. We denote the 
associated (complex) symplectic form by $\omega_\C$, and its real part by $\omega_C$. 

An equivalent way to describe the complex symplectic structure on $\CP$ is via the 
holonomy representation. This viewpoint naturally leads to a complex symplectic structure
on $\CP$, see \cite{goldman-symplectic}, defined in terms of the cup-product of two 
1-cohomology classes on $S$ with values in the appropriate Lie algebra bundle over $S$.
We call this complex symplectic structure $\omega_G$ here. 
The fact that this complex symplectic structure is the same (up to a constant) as $\omega_{\C}$ 
was established by Kawai \cite{kawai}.

Note that Kawai \cite{kawai} uses another way to associate a holomorphic quadratic
differential to a complex projective structure on $S$, using as a ``reference
point'' a complex projective structure given by the simultaneous uniformization 
(Bers slice) instead of the Fuchsian structure $\sigma_0$. This identification is
not as canonical as the one above, as it depends on a chosen reference conformal
structure needed for the simultaneous uniformization. It turns out that the symplectic structure 
obtained in this way on $\CP$ is independent of the reference point and is the same as the 
one coming from the above construction using the Fuchsian projective structure $\sigma_0$,
see Lemma \ref{lm:fuchsian-bers}.

\subsection{The grafting map}

The ``hyperbolic'' and the ``complex'' descriptions of Teichm\"uller space
behave differently in some key aspects, and it is interesting to understand
the relation between them. Such a relation is given by the well-known grafting map 
$Gr:\cT_H\times \ML\mapsto \CP$.

The grafting map is defined as follows.
When $m\in \cT_H$ is a hyperbolic metric and $l\in \ML$ is a weighted multicurve, 
$Gr_l(m)$ can be obtained by ``cutting'' $(S,m)$ open along the leaves of $l$, gluing in 
each cut a flat cylinder of width equal to the weight of the curve in $l$, and considering 
the complex projective structure underlying this metric. This map extends by continuity 
from weighted multicurves to measured laminations, a fact discovered by Thurston, 
see e.g. \cite{dumas-survey}. Out of $Gr$ one can obtain a map from the
Teichm\"uller space to itself by fixing a measured lamination $l\in\ML$ and reading the conformal
structure underlying $Gr_l(m)$, this map is known to be a homeomorphism,
see \cite{scannell-wolf}. Grafting on a fixed hyperbolic surface also defines a 
homeomorphism between $\ML$ and $\cT$, see \cite{dumas-wolf}.

It is possible to compose $\delta^{-1}:T^*\cT_H\rightarrow \cT_H\times \ML$ with the grafting
map $Gr:\cT_H\times \ML\rightarrow \CP$. The resulting map between smooth manifolds is
tangentiable by the results mentioned above, but turns out to be smoother.

\begin{lemma} \label{lm:C1}
$Gr\circ \delta^{-1}:T^*\cT_H\rightarrow \CP$ is $C^1$.
\end{lemma}

The proof is in section 2.

Our main result
in this paper is that this composed map is symplectic. This can be stated as follows, 
using the symplectic structure induced on $\cT\times \ML$ by $\omega_H$.

\begin{thm} \label{tm:main}
The pullback of the symplectic form $\omega_C$ on $\CP$ by the grafting map is the 
form $\omega_H$ on $\cT_H\times \ML$, up to a factor of $2$: $Gr^*\omega_C=2\omega_H$. 
\end{thm}

The meaning of the theorem is easier to see when considering the composed map
$Gr\circ\delta^{-1}:T^*\cT_H\rightarrow \CP$. Since this map is $C^1$ by Lemma
\ref{lm:C1}, one can use it to pull back the symplectic form $\omega_C$.

What the proof shows is that the image by $Gr$ of the
Liouville form of $2\omega_H$ is the Liouville form of $\omega_C$ plus the 
differential of a function. 
Below we shall give an alternative statement
of the above theorem in terms of Lagrangian submanifolds. 

Our proof of Theorem \ref{tm:main} is based on geometrically
finite 3-dimensional hyperbolic ends. We recall this notion in the next subsection.

\subsection{Hyperbolic ends}

\begin{df} \label{df:end}
A {\bf hyperbolic end} is a 3-manifold $M$, homeomorphic to $S\times \R_{> 0}$,
where $S$ is a closed surface of genus at least $2$, endowed with a (non-complete)
hyperbolic metric such that:
\begin{itemize}
\item the metric completion corresponds to $S\times \R_{\geq 0}$, 
\item the metric $g$ extends to a hyperbolic metric in a 
neighbourhood of the boundary, in such a way
that $S\times \{ 0\}$ corresponds to a pleated surface, 
\item $S\times \R_{> 0}$ is concave in the neighbourhood of this boundary.
\end{itemize}
Given such a hyperbolic end, we call $\dr_0M$ the ``metric'' boundary corresponding
to  $S\times \{ 0\}$, and $\dr_\infty M$ the boundary at infinity.
We call $\cG_S$ the space of those hyperbolic ends.
\end{df}

It is simpler to consider a quasifuchsian hyperbolic manifold $N$. The 
complement of its convex core is the disjoint union of two hyperbolic ends.
However a hyperbolic end, as defined above, does not always extend to a
quasifuchsian manifold. Note also that the hyperbolic ends as defined here
are always convex co-compact, so our definition is more restrictive
than others found elsewhere, and the longer name ``convex co-compact
hyperbolic end'' would perhaps be more precise. We do not consider here 
degenerate hyperbolic ends, with an end invariant which is a lamination rather
than a conformal structure; the fact that $S\times \{ 0\}$ is a convex
pleated surface in our definition prevents the other end from being degenerate
at infinity.

There are two natural ways to describe a hyperbolic end, either from the metric
boundary or from the boundary at infinity, both of which are well-known. 
On the metric boundary side, $\dr_0M$
has an induced metric $m$ which is hyperbolic, and is pleated along a measured
lamination $l$. It is well known that $m$ and $l$ uniquely determine $M$, see 
e.g. \cite{dumas-survey}. 

In addition, $\dr_\infty M$ carries naturally a complex projective structure,
$\sigma$,
because it is locally modelled on the boundary at infinity of $H^3$ and
hyperbolic isometries act at infinity by M\"obius transformations. This
complex projective structure has an underlying conformal structure, $c$.
Moreover the construction described above assigns to 
$\dr_\infty M$ a quadratic holomorphic differential $q$, which is none other
than the Schwartzian derivative of the complex map from $(S,\sigma_0)$
to $(S,\sigma)$. It follows from Thurston's original construction of the
grafting map that $\sigma=Gr_l(m)$.

\subsection{Convex cores}

Before we describe how the above hyperbolic ends can be of any use for
the questions considered in this paper, let us consider what is
perhaps a more familiar situation. Thus, consider a hyperbolic 3-manifold with 
boundary $N$, which admits a 
convex co-compact hyperbolic metric. We call $\cG(N)$ the space of such convex
co-compact hyperbolic metrics on $N$. Let $g\in \cG$, then $(N,g)$ 
contains a smallest non-empty subset $K$ which
is geodesically convex (any geodesic segment with endpoints in $K$
is contained in $K$), its convex core, denoted here by $CC(N)$. 
$CC(N)$ is then homeomorphic 
to $N$, its boundary is the disjoint union of closed pleated surfaces, 
each of which has an induced metric which is hyperbolic, and each is
pleated along a measured geodesic lamination, see e.g. \cite{epstein-marden}. 
So we obtain a map
$$ i':\cG(N)\rightarrow \cT_H(\dr N)\times \ML(\dr N)~. $$
Composing $i'$ with the identification $\delta$ between $\cT_H\times \ML$
and $T^*\cT_H$, we obtain an injective map
$$ i : \cG(N)\rightarrow T^*\cT_H(\dr N)~. $$
 
\begin{thm} \label{tm:cc}
$i(\cG(N))$ is a Lagrangian submanifold of $(T^*\cT_H(\dr N),\omega_H)$.
\end{thm}

The map $i$ is not smooth, but as the map
$\delta$ defined above it is tangentiable (see \cite{bonahon,bonahon2}). 
The natural map from $\cG(N)$ to the space of complex projective structure
on each connected component of the boundary at infinity is smooth, and 
it follows from Theorem \ref{tm:bonahon} that $i$ is $C^1$.

The proof given below shows that
the restriction to $i(\cG(N))$ of the Liouville form of $T^*\cT_H(\dr N)$ is
the differential of a function.

The reason for considering convex cores in our context will become clear in the next
two subsections.

\subsection{Kleinian reciprocity}

There is a direct relationship between Theorem \ref{tm:cc} and Theorem \ref{tm:main},
in that Theorem \ref{tm:cc} can be considered as a corollary of Theorem \ref{tm:main}.
This goes via the so-called ``Kleinian reciprocity'' of McMullen. Thus, 
consider a Kleinian manifold $M$, and let $\cG(M)$ be the space of 
complete convex co-compact hyperbolic metrics on $M$. Each $g\in\cG(M)$ gives rise
to a projective structure on the boundary at infinity
$\partial_\infty M$. This gives an injective map $j:\cG(M)\rightarrow T^*\cT_C(\dr_\infty M)$.

\begin{thm}[McMullen \cite{McMullen}] \label{tm:mcm}
$j(\cG(M))$ is a Lagrangian submanifold of $(T^*\cT_C(\dr_\infty M),\omega_C)$. 
\end{thm}

This statement is quite analogous to Theorem \ref{tm:cc}, with the only difference being
that the ``hyperbolic'' cotangent bundle at boundaries of the convex core
is replaced by the ``complex'' one. This statement (under a different formulation)
is proved in the appendix of 
\cite{McMullen} under the name of ``Kleinian reciprocity'', and is an
important technical statement allowing the author to prove the K\"ahler
hyperbolicity of Teichm\"uller space. 

Let us note that Theorem \ref{tm:cc} is a direct
consequence of Theorem \ref{tm:mcm} and of Theorem \ref{tm:main}. This will become
more clear below when we present another statement of Theorem \ref{tm:main}. 
Below we will give a direct proof of Theorem \ref{tm:cc}, thus also
giving a more direct proof of the Kleinian reciprocity result.

Using the result of Kawai \cite{kawai} already mentioned above, 
Theorem \ref{tm:mcm} is equivalent
to the fact that the subspace of complex projective structures on $\dr N$
obtained from hyperbolic metrics on $N$ is a Lagrangian submanifold
of $(\CP(\dr N), Re(\omega_G))$, a fact previously known to Kerckhoff through
a different, topological argument involving Poincar\'e duality (personal 
communication).

\subsection{A Lagrangian translation of Theorem \ref{tm:main}}

In a similar vein
to what we have done above, let us consider the space  $\cG_S$ of hyperbolic ends.
Each such space gives a point in $\cT_H\times \ML$ for its pleated surface boundary,
and a point in $T^*\cT_C$ for its boundary at infinity. Thus, composing this
with the map $\delta$ we get an injective map:
$$k: \cG\to T^*\cT_H\times T^*\cT_C~.$$
Our main Theorem \ref{tm:main} can then be restated as follows:

\begin{thm}
$k(\cG)$ is a Lagrangian submanifold of $T^*\cT_H\times T^*\cT_C$.
\end{thm}

We will actually prove our main result in this version, which is 
clearly equivalent to Theorem \ref{tm:main}. Let us stress again that 
$k(\cG)$ is not smooth, but only the graph of a $C^1$ map.

\subsection{The intersection form and $\omega_H$} \label{ssc:intersection}

An efficient ``combinatorial'' description of $\cT_H$ was given in the work of Thurston
on earthquakes. Later, a powerful analytical realization of the same ideas was developed
in a series of papers by Bonahon \cite{bonahon-toulouse}, \cite{bonahon-ens}. The
earthquake description of $\cT_H$ is somewhat related to a much earlier parametrization
of the same space in terms of the Fenchel-Nielsen coordinates, whose idea
is to glue a Riemann surface from pairs of pants, the pants being characterized by the
length of their boundary components and the gluing being characterized by ``twists
parameters''. 
Thurston's earthquake description of $\cT_H$ 
describes a hyperbolic metric $m\in \cT_H$ as obtained by a left earthquake on
a measured lamination from another base hyperbolic metric $m_0$. It is remarkable
that this measured lamination completely determines the earthquake, and is in turn
completely determined by the two metrics $m, m_0\in \cT_H$. 

Thus, in Thurston's description the hyperbolic Teichm\"uller space is parametrized
by the space of measured geodesic laminations. However, the space $\ML$ does not
possess a natural differentiable structure, which makes the analysis on this space
hardly possible. One of the key achievements of the works of Bonahon \cite{bonahon-toulouse}, 
\cite{bonahon-ens} was to develop the calculus on $\ML$ using $\R$-valued transverse
cocycles or, equivalently, transverse H\"older distributions for geodesic laminations.
Essentially, Bonahon gave a very elegant description of the tangent space to $\ML$.
This allowed him to provide a characterization of the space $\cT_H$ itself as
homeomorphic to an open cone in a vector space $\cH(\lambda, \R)$ of $\R$-valued transverse cocycles for
a lamination $\lambda$, see Theorem A in \cite{bonahon-toulouse},
and also prove that the vector fields tangent to $\ML$ are Hamiltonian vector
fields with respect to Thurston's symplectic structure on $\cH(\lambda, \R)$,
with the Hamiltonian function being essentially the hyperbolic length, see
in \cite{bonahon-toulouse}. 

In a later work Bonahon and S\"ozen \cite{sozen-bonahon} established
that the Thurston's symplectic form on $\cH(\lambda, \R)$ is (up to a
constant) an image of the Weil-Petersson symplectic form on $\cT_H$ under
the homeomorphism of this space into $\cH(\lambda, \R)$. The proof goes
through Goldman's characterization of the Weil-Petersson symplectic form 
in terms of a cup product in a twisted cohomology group \cite{goldman-symplectic}.
Thus, Bonahon's description of $\cT_H$ in terms of shearing coordinates 
can be said to provide a symplectic map from $\cT_H$ with its usual Weil-Petersson
symplectic form to the vector space $\cH(\lambda, \R)$ with its Thurston's
symplectic form. Theorem \ref{tm:main} can be construed as an analog 
concerning related but twice bigger spaces: on one hand $\cT_H \times \ML$, on the other 
the space $\CP$ of complex projective structures on $S$.

The space $\cT_H \times \ML$ is naturally a real symplectic manifold. The
length function has an extension 
$L_m: \cH(\lambda, \R)\to \R, \dot{l}\in \cH(\lambda, \R)\to L_m(\dot{l})\in \R$ to geodesic 
laminations with transverse cocycles, where it can be interpreted as a differential of the 
corresponding function on $\ML$, see \cite{bonahon-toulouse}. Now, given a vector
field $\dot{m}$ tangent to $\cT_H$ we obtain a pairing $L_{\dot{m}}(\dot{l})$ as
the derivative of $L_m(\dot{l})$ in the direction of $\dot{m}$. Consider the following
two-form on $\cT_H \times \ML$:
\be
\omega'_H((\dot{m}_1,\dot{l}_1),(\dot{m}_2,\dot{l}_2)) = L_{\dot{m}_1}(\dot{l}_2)
-L_{\dot{m}_2}(\dot{l}_1).
\ee

\begin{lemma}
$\omega'_H=\omega_H$.
\end{lemma}

\begin{proof}
Let $(m,l)\in \cT\times \ML$, then $\delta(m,l)=(m,dL_\cdot(l))\in T^*\cT_H$. 
Denote by $\beta$ the Liouville form of $(T^*\cT_H,\omega_H)$, so that, if $(m,u)\in T^*\cT_H$
and $(\md, \ud)\in T_{(m,u)}T^*\cT_H$, then 
$$ \beta(\md,\ud)=u(\md)~. $$
If $(\md,\ld)\in T(\cT\times\ML)$ then $\beta(d\delta(\md, \ld))=L_{\md}(l)$.
This corresponds precisely to the Liouville form of $\omega'_H$, the result follows.
\end{proof}

\subsection{Another possible proof ?} \label{ssc:another}

Let us note that an alternative proof that only uses 2-dimensional quantities
may be possible, based essentially on the result of \cite{sozen-bonahon}.\footnote{%
We thank one of the anonymous referees for suggesting such an alternative proof.}
For this one would need to extend the ideas developed in this work to shear-bend
coordinates on $\cT_H\times \ML$, and show that the symplectic form $\omega_H$ on 
$\cH(\lambda,\R)$ extends to a complex symplectic form on the vector space of
complex-valued cocycles for $\lambda$. Presumably this complex symplectic form
would then coincide with the complexified Thurston symplectic form on $\cH(\lambda,\C)$. 

According to \cite{sozen-bonahon}, Thurston's symplectic form on $\cH(\lambda,\R)$ 
is equal to the Weil-Petersson symplectic form on the real character variety.
This equality should extend to the complexified Thurston symplectic form and 
the Weil-Petersson symplectic form on the complex character variety, because
both are holomorphic. 

But it was proved by Kawai \cite{kawai} that the Weil-Petersson form on the complex
character variety corresponds to the complex symplectic cotangent symplectic
form on $T^*\cT_C$, and Theorem \ref{tm:main} should follow.

Note that this line of reasoning is quite different from the the proof considered
below, which uses mostly the Bonahon-Schl\"afli formula in Lemma \ref{lm:bs} (or more
precisely the dual formula in Lemma \ref{lm:bs-dual}). Therefore the arguments
outlined in this subsection could be combined with those developed below, for
instance to obtain a new proof of Kawai's result \cite{kawai} from the results
of \cite{sozen-bonahon}, holomorphic continuation, and the use of the renormalized
volume.

\subsection{Cone singularities}

One interesting feature of the arguments used here is that they appear
likely to extend to the setting of hyperbolic surfaces with cone
singularities of angle less than $\pi$. One should then use hyperbolic
ends with ``particles'', i.e., cone singularities of angle less than $\pi$
going from the ``interior'' boundary to the boundary at infinity, as already
done in \cite{minsurf} and to some extent in \cite{volume}.

\subsection*{Acknowledgements}

The authors would like to thank Francis Bonahon for very useful and
relevant comments, in particular for the arguments used in the proof
of Lemma \ref{lm:bs-dual}. We are also grateful to St\'ephane 
Baseilhac, David Dumas and to 
Steve Kerckhoff for useful conversations related to the questions
considered here. The text was notably improved thanks to remarks
from two anonymous referees, one of which suggested the content
of subsection \ref{ssc:another}.

\section{The Schl\"afli formula and the dual volume}

In this section we recall the Schl\"afli formula, first in the simple
case of hyperbolic polyhedra, then in the more involved setting of convex
cores of hyperbolic 3-manifolds (as extended by Bonahon). We then deduce
from Bonahon's Schl\"afli formula a ``dual'' formula for the first-order
variation of the dual volume of the convex core. Finally we give the proof
of Lemma \ref{lm:C1}.

\subsection{The Schl\"afli formula for hyperbolic polyhedra}

Let $P\subset H^3$ be a convex polyhedron. The Schl\"afli formula
(see e.g. \cite{milnor-schlafli}) describes the first-order variation of
the volume of $P$, under a first-order deformation, in terms of the 
lengths and the first-order variations of the angles, as follows:
\begin{equation}\label{eq:schlafli}
dV = \frac{1}{2} \sum_e L_e d\theta_e~, 
\end{equation}
where the sum is over the edges of $P$, $L_e$ is the length of the
edge $e$, and $\theta_e$ is its exterior dihedral angle. 

There is also an interesting ``dual'' Schl\"afli formula.
Let 
$$ V^* = V - \frac{1}{2}\sum_e L_e\theta_e~, $$
be the {\it dual volume} of $P$, then, still under a first-order
deformation of $P$,
\begin{equation}\label{eq:schlafli-dual}
dV^* = - \frac{1}{2} \sum_e \theta_e dL_e~. 
\end{equation}
This follows from the Schl\"afli formula (\ref{eq:schlafli})
by an elementary computation.

\subsection{First-order variations of the volume of the convex core}

In many ways, the convex core of a quasifuchsian manifold is reminiscent 
of a polyhedron, with the edges and their exterior
dihedral angles being replaced by a measured lamination describing
the pleating of the boundary, see e.g. \cite{thurston-notes,epstein-marden}. 

Bonahon \cite{bonahon} has extended the Schl\"afli formula to 
this setting as follows. Let $M$ be a convex co-compact hyperbolic
manifold (for instance, a quasifuchsian manifold), let $\mu$
be the induced metric on the boundary of the convex core, and let 
$\lambda$ be its measured bending lamination. By a ``first-order
variation'' of $M$ we mean a first-order variation of the 
representation of the fundamental group of $M$. Bonahon shows that
the first-order variation of $\lambda$  under a first-order
variation of $M$ is described by a 
transverse H\"older distribution $\lambda'$, and there is a 
well-defined notion of length of such transverse H\"older 
distributions. This leads to 
a version of the Schl\"afli formula.

\begin{lemma}[The Bonahon-Schl\"afli formula \cite{bonahon}]\label{lm:bs}
The first-order variation of the volume $V_C$ of the convex
core of $M$, under a first-order variation of $M$, is given by
$$ dV_C = \frac{1}{2} L_\mu(\lambda')~. $$
\end{lemma}

Here $\lambda'$ is the first-order variation of the measured bending lamination,
which is a H\"older cocycle so that its length for $\mu$ can be defined, see 
\cite{bonahon-toulouse,bonahon-ens,bonahon,bonahon2}. 

\subsection{The dual volume}

Just as for polyhedra above, we define the dual volume
of the convex core of $M$ as 
$$ V_C^* = V_C -\frac{1}{2} L_\mu(\lambda)~. $$

\begin{lemma}[The dual Bonahon-Schl\"afli formula] \label{lm:bs-dual}
The first-order variation of $V^*$ under a first-order 
variation of $M$ is given by
$$ dV_C^* = - \frac{1}{2} L'_\mu(\lambda)~. $$
\end{lemma}

This formula has a very simple interpretation in terms of 
the geometry of Teichm\"uller space: up to the factor $-1/2$,
$dV^*$ is equal to the pull-back by $\delta$ of the Liouville form of the cotangent
bundle $T^*\cT_H$. Note also that this formula can be 
understood in an elementary way, without reference to a
transverse H\"older distribution: the measured lamination $\lambda$ is
fixed, and only the hyperbolic metric $\mu$ varies. 
The proof we give here,
however, is based on Lemma \ref{lm:bs} and thus on the
whole machinery developed in \cite{bonahon}.

Theorem \ref{tm:cc} is a direct consequence of Lemma 
\ref{lm:bs-dual}: since $dV_C^*$ coincides with the
Liouville form of $T^*\cT_H(\dr N)$ on $i(\cG(N))$, it follows
immediately that $i(\cG(N))$ is Lagrangian for the symplectic
form $\omega_H$ on $T^*\cT_H(\dr N)$.

\begin{proof}[Proof of Lemma \ref{lm:bs-dual}]
Thanks to Lemma \ref{lm:bs} we only have to show a purely 2-dimensional statement,
valid for any closed surface $S$ of genus at least $2$: 
that the function
$$ \begin{array}{cccc}
L: & \cT \times \ML & \rightarrow & \R\\
& (\mu,\lambda) & \mapsto & L_\mu(\lambda)
\end{array}
$$
admits directional derivatives, and that its derivative with respect to a 
tangent vector $(\mu', \lambda',)$ is equal to 
\begin{equation}
  \label{eq:der}
 L_\mu(\lambda)' = L_\mu'(\lambda) + L_\mu(\lambda')~.  
\end{equation}
Two special cases of this formula were proved by Bonahon: when $\mu$ is
kept constant \cite{bonahon-ens} and when $\lambda$ is kept constant 
\cite{bonahon-toulouse}.

To prove equation (\ref{eq:der}), suppose that $\mu_t, \lambda_t$ depend on 
a real parameter $t$ chosen so that the derivatives $\mu'_t, \lambda'_t$ exist
for $t=0$, with
$$ \frac{d\mu_t}{dt}_{|t=0} = \mu'~,  ~~ \frac{d\lambda_t}{dt}_{|t=0} = \lambda'~. $$
We can also suppose that $(m_t)$ is a smooth curve for the differentiable
structure of Teichm\"uller space. We can then decompose as follows :
$$ \frac{L_{\mu_t}(\lambda_t)-L_{\mu_0}(\lambda_0)}{t} = 
\frac{L_{\mu_t}(\lambda_t) - L_{\mu_0}(\lambda_t)}{t} + 
\frac{L_{\mu_0}(\lambda_t)-L_{\mu_0}(\lambda_0)}{t}~. $$
The second term on the right-hand side converges to $L_\mu(\lambda')$ by 
\cite{bonahon-ens} so we now concentrate on the first term. 

To prove that the first term converges to $L'_\mu(\lambda)$, it is sufficient
to prove that $L'_\mu(\lambda)$ depends continuously on $\mu, \mu'$ and on 
$\lambda$. This can be proved by a nice and simple argument, which was suggested
to us by Francis Bonahon. $\mu$ can be replaced by a representation of the fundamental
group of $S$ in $PSL_2(\C)$, as in \cite{bonahon-toulouse}. For fixed $\lambda$,
the function $\mu\rightarrow L_\mu(\lambda)$ is then holomorphic in $\mu$, and continuous in 
$\lambda$. Since it is holomorphic, it is continuous with respect to $\mu$ and to $\mu'$,
and the result follows.
\end{proof}

\subsection{A cotangent space interpretation}

Here we sketch for completeness the argument showing that the map
$\delta:\cT_H\times \ML
\rightarrow T^*\cT_H$ defined in the introduction is a homeomorphism. This is
equivalent to the following statement.

\begin{lemma}\label{lm:delta}
Let $m_0\in \cT_H$ be a hyperbolic metric on $S$. For each cotangent vector 
$u\in T^*_{m_0}\cT_H$, there exists a unique $l\in \ML$ such that the differential
of the function $m\mapsto dL_m(l)$ is equal to $u$ at $m_0$.  
\end{lemma}

\begin{proof}
Wolpert \cite{wolpert-formula} discovered that the Weil-Petersson symplectic
form on $\cT_H$ has a remarkably simple form in Fenchel-Nielsen coordinates: 
$$ \omega_{WP} = \sum_i dL_i \wedge d\theta_i~, $$
where the sum is over the simple closed curves in the complement of a pants
decomposition of $S$. 
A direct consequence is that, given a weighted multicurve $w$ on $S$, the 
dual for $\omega_{WP}$ of the differential of the length $L_w$ of $w$
is equal to the infinitesimal fractional Dehn twist along $w$. 

This actually extends when $w$ is replaced by a measured lamination $\lambda$,
with the infinitesimal fractional Dehn twist replaced by the earthquake vector
along $\lambda$, see \cite{wolpert-ajm,sozen-bonahon}. So the Weil-Petersson
symplectic form provides a duality between the differential of the lengths 
of measured laminations and the earthquake vectors. 

Moreover the earthquake vectors associated to the elements of $\ML$ cover
$T_m\cT_H$ for all $m\in \cT_H$ (see \cite{kerckhoff}), it follows that the
differentials of the lengths of the measured laminations cover $T^*_m\cT_H$.
\end{proof}

Note that this argument extends directly to hyperbolic surfaces with cone singularities,
when the cone angles are less than $\pi$. In that case the fact that earthquake vectors
still span the tangent to Teichm\"uller space follows from \cite{cone}.

\subsection{Proof of Lemma \ref{lm:C1}}

Lemma \ref{lm:C1} is mostly a consequence of the tools developed by Bonahon
in \cite{bonahon-toulouse,bonahon-variations}. We first recall some of
his results. Given a lamination $\lambda$ on $S$, he defined the space
$\cH(\lambda,\R)$ of real-valued transverse cocycles for $\lambda$, and
proved that it is related to measured laminations in interesting ways.
\begin{itemize}
\item If $l\in \ML$ and 
$\lambda$ is a lamination which contains the support of $l$, then $l$ defines a real-valued 
transverse cocycle on $\lambda$ (see \cite{bonahon-toulouse}). 
\item Transverse cocycles can be used to define a polyhedral ``tangent cone'' to $\ML$ at a point
$l$. Given a lamination $\lambda$ containing the support of $l$, the transverse cocycles on $\lambda$
satisfying a positivity condition (essentially, that the transverse measure remains positive)
can be interpreted as ``tangent vectors'' to $\ML$ at $l$, i.e., velocities at $0$
of curves in $\ML$ starting from $l$. The laminations containing the support of $l$ therefore 
correspond to the ``faces'' of the tangent cone to $\ML$ at $l$.
\item There is a well-defined notion of length of a transverse cocycle $h$ for a 
hyperbolic metric $m$ on $S$, extending the length of a measured lamination. 
If $l\in \ML$ then $L_m$ is tangentiable at $l$, if $\lambda$ is a lamination 
containing the support of $l$, and if $h\in \cH(\lambda,\R)$, then $L_m(h)$ is equal
to the first-order variation of $L_m(l)$ under the deformation of $l$ given by $h$.
\end{itemize}

Transverse cocycles are also related to pleated surfaces. 
\begin{itemize}
\item Transverse cocycle provide ``shear coordinates'' on Teichm\"uller space. 
Given a ``reference'' hyperbolic metric $m_0\in \cT_H$ and another 
hyperbolic metric $m\in \cT_H$, there is a unique element $h_0\in \cH(\lambda,\R)$ such 
that ``shearing'' $m_0$ along $h$ yields $m$. The elements of $\cH(\lambda,\R)$ which
can be obtained in this way have a simple characterization in terms of a positivity
condition. 
\item Transverse cocycles also describe the bending of a pleated surface: $\cH(\lambda,
\R/2\pi \Z)$ is in one-to-one correspondence with the space of equivariant
pleated surfaces of given induced metrics for which the support of the pleating 
locus is contained in $\lambda$.
\item Pleated surfaces with pleating locus contained in $\lambda$ 
are associated to a complex-valued
transverse cocycle $h\in \cH(\lambda,\C/2\pi i\Z)$, with real part describing
the induced metric (in terms of its ``shear coordinates'' with respect to a given
reference metric) and imaginary part describing the bending measure.
\end{itemize}

Each pleated equivariant surface in $H^3$ defines a representation of its fundamental
group in $PSL(2,\C)$. In the neighbhorhood of a convex pleated surface, this
representation is the holonomy representation of a complex projective structure.
If the induced metric and measured bending lamination of the convex pleated surface
are $m\in \cT_H$ and $l\in \ML$ respectively, if $\lambda$ contains the 
support of $l$, and if $\lb$ is the projection of $l$ in 
$\cH(\lambda, \R/2\pi \Z)$, there is a well-defined map from
$\cH(\lambda,\C/2\pi i\Z)$ to $\CP$ defined in the neighborhood of $i\lb$, sending
a complex-valued transverse cocycle $h$ to the complex projective structure $\sigma$
of the pleated surface obtained from $h$. This map is differentiable, taking
its tangent at $i\lb$ yields a map
$$ \phi_\lambda:\cH(\lambda,\C)\rightarrow T_\sigma \CP~. $$

\begin{thm}[Bonahon \cite{bonahon-variations}] \label{tm:bonahon}
The map $\phi_\lambda$ is complex-linear (with respect to the complex structure
on $\CP$).  
\end{thm}

However, a pleated surface is also described by its induced metric and 
measured bending lamination, and thus an element of $\cT_H\times \ML$. Using
the map $\delta:\cT_H\times \ML\rightarrow T^*\cT_H$ defined above, we obtain
a map, defined in the neighborhood of $i\lb$, from $\cH(\lambda,\C/2\pi i\Z)$ to
$T^*\cT_H$, which by definition is also differentiable.
Taking the differential of this map yields another linear map
$$ \psi_\lambda:\cH(\lambda,\C)\rightarrow T^*\cT_H~. $$

The definitions (and the arguments of \cite{bonahon-toulouse,bonahon-variations})
then show that $\phi_\lambda\circ\psi_\lambda^{-1}$ is partially equal to the 
tangent map of $Gr\circ \delta^{-1}$, in the following sense. Let $m\in \cT_H$,
and let $u\in T^*_m\cT_H$, and let $(m,l)=\delta^{-1}(m,u)\in \cT_H\times \ML$. 
Let then $(\md,\ud)\in T_{(m,u)}(T^*\cT_H)$, and let $\ld$ be the tangent vector
to $\ML$ at $l$ corresponding to $\ud$. There is then a lamination $\lambda$ 
containing the support of both $l$ and $\ld$, and 
$\phi_\lambda\circ\psi_\lambda^{-1}(\md,\ud)=T(Gr\circ \delta^{-1})(\md,\ud)$. 
Its definition shows that $\phi_\lambda\circ\psi_\lambda^{-1}$ is linear. So,
to prove that $Gr\circ \delta^{-1}$ is differentiable, it is sufficient to
show that $\phi_\lambda\circ\psi_\lambda^{-1}$ does not depend on $\lambda$.

We now consider a fixed lamination $\lambda$ containing the support of 
$l$, and a variation $(\md,\ud)$ inducing a variation $\ld$ of $l$ with support
contained in $\lambda$. To $(\md,\ld)$ is associated, through Bonahon's 
shear-bend coordinates, a complex transverse cocycle $h\in\cH(\lambda,\C)$,
with real part $h_0$ corresponding to $\md$ and imaginary part corresponding
to $\ld$. The first-order variation of $\sigma$ corresponding to $\md$,
$\phi_\lambda\circ\psi_\lambda^{-1}(\md,0)$, is shown in \cite{bonahon-variations}
to be independent of $\lambda$, actually the following is clearly equivalent to 
Lemma 13 in \cite{bonahon-variations}.

\begin{lemma}[Bonahon] \label{lm:TT}
For fixed $l\in \ML$, the restriction map $Gr_l:\cT\rightarrow \CP$ is $C^1$.  
\end{lemma}

We now focus on $\ld$, and on the corresponding imaginary
part $h_1$ of the transverse cocycle $h$. 
We have already recalled that $L_m(h_1)=L_m(l)'$, where the prime denotes the
first-order variation under the tangent vector to $\ML$ at $l$ corresponding to $h_1$.
It follows that for any $m'\in T_m\cT_H$, $\ud(m')=dL_\cdot(h_1)(m')$. 
However it was proved in \cite{bonahon-toulouse} that 
$dL_\cdot(h_1)(m')=\omega_{WP}(m',e_{h_1})$, where $\omega_{WP}$ is the Weil-Petersson
symplectic form on $\cT_H$ and $e_{h_1}$ is the tangent vector to $\cT_H$ at $m$ 
corresponding to the infinitesimal shear along $h_1$. 
So $e_{h_1}$ is the dual of 
$\ud$ for the symplectic form $\omega_H$ on $T^*\cT_H$, we write this 
as $e_{h_1}=\ud^*$ (the star stands for the Weil-Petersson symplectic duality).

We can now apply Theorem \ref{tm:bonahon}, and conclude that 
\begin{equation}
  \label{eq:comp}
 (\phi_\lambda\circ\psi_\lambda^{-1})(0,\ud) = \phi_\lambda(ih_1)=i\phi_\lambda(h_1)
= i(\phi_\lambda\circ\psi_\lambda^{-1})(e_{h_1},0) = 
i(\phi_\lambda\circ\psi_\lambda^{-1})(\ud^*,0)=idGr_l(\ud^*)~.
\end{equation}
In particular, $(\phi_\lambda\circ\psi_\lambda^{-1})(0,\ud)$ is independent of $\lambda$ by
Lemma \ref{lm:C1},
so that $\phi_\lambda\circ\psi_\lambda^{-1}$ is linear, and therefore $Gr\circ \delta^{-1}$
is differentiable.

The fact that $Gr\circ \delta^{-1}$ is actually $C^1$ then follows from 
Theorem \ref{tm:bonahon} applied twice, once for the first-order variations
of the metric and another time, through the composition (\ref{eq:comp}),
for the first-order variation of $\ud$ (resp. $\ld$).

Note that this map $Gr\circ \delta^{-1}$ is probably not $C^2$. This is indicated
by the fact, shown by Bonahon in \cite{bonahon-variations}, that the composition
of the inverse grafting map $Gr^{-1}:\CP\rightarrow \cT_H\times \ML$ with the
projection on the first factor is $C^1$ but not $C^2$.

\section{The renormalized volume}

\subsection{Definition}
\label{ssec:renormalized}

We recall in this section, very briefly, the definition and one
key property of the renormalized volume of a quasifuchsian -- or
more generally a geometrically finite -- hyperbolic 3-manifold; more
details can be found in \cite{volume}. 
The definition can be made as follows. Let $M$ be a quasifuchsian
manifold and let $K$ be a compact subset which is {\it geodesically
convex} (any geodesic segment with endpoints in $K$ is contained
in $K$), with smooth boundary.

\begin{df}
We call
$$ W(K) = V(K) - \frac{1}{4}\int_{\dr K} H da~, $$
where $H$ is the mean curvature of the boundary of $K$. 
\end{df}

Actually 
$K$ defines a metric $I^*$ on the boundary of $M$. For $\rho>0$, let 
$S_\rho$ be the set of points at distance $\rho$ from $K$, then
$(S_\rho)_{\rho>}$ is an equidistant foliation of $M\setminus K$. It is then
possible to define a metric on $\dr M$ as
\begin{equation} \label{eq:I*}
I^* := \lim_{\rho\rightarrow \infty} 2e^{-2\rho}I_\rho~,
\end{equation}
where $I_\rho$ is the induced metric on $S_\rho$. Then $I^*$ is in the
conformal class at infinity of $M$, which we call $c_\infty$. 

Defined in this way, both $I^*$ and $W$ are functions of the convex
subset $K$. However $K$ is itself uniquely determined by $I^*$, and it 
is possible to consider $W$ as a function of $I^*$, considered as a 
metric in $\dr M$ in the conformal class at infinity $c_\infty$, although
such a metric in $c_\infty$ is not necessarily associated to a convex subset
of $M$. The reason for this is that each metric $I^*\in c_\infty$ is associated
to a unique foliation of a neighborhood of infinity in $M$ by equidistant 
convex surfaces $(S_\rho)_{\rho\geq \rho_0}$, see \cite{Eps,c-epstein,horo} or
Theorem 5.8 in \cite{volume}
(this foliation does not always extend
to $\rho\rightarrow 0$, which would mean that it is the equidistant foliation 
from a convex subset with boundary $S_0$). 

To understand the construction of $W$ in this setting we need to revert to 
another definition of the renormalized volume as it is defined for higher-dimensional
conformally compact Einstein manifolds. If $V_\rho$ is the volume of the set of
points of $M$ at distance at most $\rho$ from $K$, then $V_\rho$ behaves as 
$\rho\rightarrow \infty$ as
$$ V_\rho = V_2e^{2\rho} + V_1\rho + V_0 + \epsilon(\rho)~, $$
where $\lim_0\epsilon=0$. It is proved by C. Epstein in \cite{epstein-duke} 
(see also \cite{minsurf}) that $V_0=W$ (as defined above) is equal to $V_0$,
while $V_1$ depends only on the topology of $M$ (it is equal to $-\pi \chi(\dr M)$).
Suppose now that $K$ is replaced by 
$$ K_{r} = \{ x\in M ~ | ~ d(K,x)\leq r\}~. $$
Let $\Vb_\rho$ be the volume of the set of points at distance at most $\rho$
from $K_r$, then clearly
$$ \Vb_\rho = V_{\rho+r} = V_2e^{2(\rho+r)} + V_1(\rho+r) + V_0 = (V_2e^{2r}) 
e^{2\rho} + V_1 \rho + (V_0+V_1r) + \epsilon(\rho)~, $$
so that $V_0$ is replaced by $\Vb_0=V_0+V_1r$. This means that $W$ can be
read off from any of the surfaces $V_\rho$, since, for any $\rho>0$, 
$$ W = V_\rho -\frac{1}{4}\int_{S_\rho} H da + \pi \chi(\dr M) \rho~. $$

Starting from a metric $I^*$ in the conformal class at infinity $c_\infty$,
there is an associated equidistant foliation by convex surfaces 
$(S_\rho)_{\rho\geq \rho_0}$ of a
neighborhood of infinity in $M$, and the previous formula can be used to
define $W$ even if the foliation does not extend to $\rho\rightarrow 0$.
As a consequence, $W$ defines a function, still called 
$W$, which, to any metric $I^*\in c_\infty$, associates a real number
$W(I^*)$. 

\begin{lemma}[Krasnov \cite{Holography}, Takhtajan, Teo \cite{takhtajan-teo}, 
see also \cite{TZ-schottky}]
Over the space of metrics $I^*\in c_\infty$ of fixed area, $W$ has
a unique maximum, which is obtained when $I^*$ has constant
curvature.
\end{lemma}

This, along with the Bers double uniformization theorem, defines a function 
$V_R:\cT(\dr M)\rightarrow \R$, sending a conformal structure on the boundary
of $M$ to the maximum value of $W(I^*)$ when $I^*$ is in the fixed conformal
class of metrics and is restricted to have area equal to $-2\pi\chi(\dr M)$. 
This number $V_R$ is called the renormalized volume of $M$.

\subsection{The first variation of the renormalized volume}

The first variation of the renormalized volume involves a kind of Schl\"afli
formula, in which some terms appear that need to be defined. One such term
is the second fundamental form at infinity $\II^*$ associated to an equidistant
foliation in a neighbourhood of infinity, as in the previous subsection. The
definition comes from the following lemma, taken from \cite{volume}. 

\begin{lemma} \label{lm:II*}
Given an equidistant foliation as above, there is a unique bilinear symmetric
2-form $\II^*$ on $\dr M$ such that, for $\rho\geq \rho_0$,
$$ I_\rho = \frac{1}{2}(e^{2\rho}I^* + 2 \II^* + e^{-2\rho}\III^*)~, $$
where $\III^*=\II^* I^{-1} \II^*$, that is, $\III^* = I^*(B^*\cdot, B^*\cdot)$
where $B^*:T\dr M\rightarrow T\dr M$ is the bundle morphism, self-adjoint for 
$I^*$, such that $\II^* = I^*(B^*\cdot, \cdot)$. 
\end{lemma}

The first variation of $W$ under a deformation of $M$ or of the equidistant
foliation is given by another lemma from \cite{volume}, which can be seen as 
a version ``at infinity'' of the Schl\"afli formula for hyperbolic manifolds
with boundary found in \cite{sem,sem-era}.

\begin{lemma} \label{lm:dW}
Under a first-order deformation of the hyperbolic metric on $M$ or of the
equidistant foliation close to infinity, the first-order variation of $W$
is given by
$$ dW = -\frac{1}{4}\int_{\dr M} \left\langle d\II^* - \frac{H^*}{2} dI^*,I^*
\right\rangle da^*~, $$
where $H^*:=\mbox{tr}(B^*)$ and $da^*$ is the area form of $I^*$. 
\end{lemma}

The ``second fundamental form at infinity'', $\II^*$, is actually quite 
similar to the usual second fundamental form of a surface. It satisfies the
Codazzi equation 
$$ d^{\nabla^*}\II^* =0~, $$
where $\nabla^*$ is the Levi-Civit\`a connection of $I^*$, as well as a 
modified form of the Gauss equation,
$$ \mbox{tr}_{I^*}(\II^*) = - K^*~, $$
where $K^*$ is the curvature of $I^*$. The proof can again be found in 
\cite{volume} (section5). A direct consequence is that, if $I^*$ has constant curvature
$-1$, the trace-less part $\II^*_0$ of $\II^*$ is the real part of a
holomorphic quadratic differential on $\dr M$ for the complex structure
of $I^*$. In addition, the first-order variation of $V_R$ follows from
Lemma \ref{lm:dW}.

\begin{lemma} \label{lm:dVR}
In a first-order deformation of $M$,
$$ dV_R = -\frac{1}{4} \int_{\dr M} \langle dI^*,\II^*_0\rangle da^*~. $$
\end{lemma}

This statement is very close in spirit to Lemma \ref{lm:bs-dual}, with the
dual volume of the convex core replaced by the renormalized volume. The
right-hand term is, up to the factor $-1/4$, the Liouville form on the
cotangent bundle $T^*\cT_C(\dr M)$. 

\begin{proof}[A simple proof of Theorem \ref{tm:mcm}]
We have just seen that $dV_R$ coincides (up to the constant $-1/4$)
with the Liouville form of
$T^*\cT_C(\dr M)$ on $j(\cG)$. It follows that the symplectic form of 
$ T^*\cT_C(\dr M)$ vanishes on $j(\cG(\dr M))$, which is precisely the statement
of the theorem.
\end{proof}


\section{The relative volume of hyperbolic ends}

\subsection{Definition}
\label{ssec:relative}

We consider in this part yet another notion of volume, defined for
(geometrically finite) hyperbolic ends rather than for hyperbolic manifolds. 
Here we consider a hyperbolic end $M$. 
The definition of the renormalized volume can be used in this setting, 
leading to the relative volume of the end. We will write that a geodesically
convex subset $K\subset M$ is a {\it collar} if it is relatively compact and 
contains the metric boundary $\dr_0M$ of $M$ (possibly all geodesically 
convex relatively compact subsets of $M$ are collars, but it is not necessary
to consider this question here). Then $\dr K\cap M$ is a locally convex surface in $M$.

The relative volume of $M$ is related both to the (dual) volume of the convex
core and to the renormalized volume; it is defined as the renormalized volume,
but starting from the metric boundary of the hyperbolic end. We follow
the same path as for the renormalized volume and start from a collar $K\subset M$. We set 
$$ W(K) = V(K) - \frac{1}{4}\int_{\dr K} H da + 
\frac{1}{2} L_\mu(\lambda)~, $$
where $H$ is the mean curvature of the boundary of $K$, $\mu$ is the 
induced metric on the metric boundary of $M$, and $\lambda$ is its
measured bending lamination.

As for the renormalized volume we define the metric at infinity as
$$ I^* := \lim_{\rho\rightarrow \infty} 2e^{-2\rho}I_\rho~, $$
where $I_\rho$ is the set of points at distance $\rho$ from $K$.
The conformal structure of $I^*$ is equal to the canonical conformal 
structure at infinity $c_\infty$ of $M$.

Here again, $W$ only depends on $I^*$. Not all metrics in $c_\infty$ can 
be obtained from a compact subset of $E$, however all metrics do define
an equidistant foliation close to infinity in $E$, and it still possible
to define $W(I^*)$ even when $I^*$ is not obtained from a convex subset of
$M$. So $W$ defines a function,
still called $W$, from the conformal class $c_\infty$ to $\R$. 

\begin{lemma}
For fixed area of $I^*$, $W$ is maximal exactly when $I^*$ has
constant curvature.  
\end{lemma}

The proof follows directly from the arguments used in \cite{volume} (section 7)
so we do not repeat the proof here. This proof takes place entirely on the
boundary at infinity so that considering a hyperbolic end or a geometrically
finite hyperbolic manifold has no impact.

\begin{df}
The relative volume $V_R$ of $M$ is  $W(I^*)$ when $I^*$ is the
hyperbolic metric in the conformal class at infinity on $M$.
\end{df}

\subsection{The first variation of the relative volume}

\begin{prop} \label{pr:relative}
Under a first-order variation of the hyperbolic end, the first-order variation
of the relative volume is given by
\begin{equation}
  \label{eq:s-rel}
  V'_R = \frac{1}{2} L'_\mu(\lambda) - \frac{1}{4} \int_{\dr_\infty E} 
\langle I^*{}',\II^*_0 \rangle da^*~.
\end{equation}
\end{prop}

The proof is based on the arguments described above, both for the
first variation of the renormalized volume and for the first
variation of the volume of the convex core. Some preliminary 
definitions are required.

\begin{df}
A {\bf polyhedral collar} in a hyperbolic end $M$ is a collar $K\subset M$
such that $\dr K\cap M$ is a polyhedral surface.
\end{df}

\begin{lemma} \label{lm:collar}
Let $K$ be a polyhedral collar in $M$, let $L_e,\theta_e$ be the length and
the exterior dihedral angle of edge $e$ in $\dr K\cap M$. In any deformation
of $E$, the first-order variation of the measured bending lamination on the
metric boundary of $M$ is given by a transverse H\"older distribution 
$\lambda'$. 
The first-order variation of the
volume of $K$ is given by 
$$ 2V' = \sum_e L_e d\theta_e - L_\mu(\lambda')~. $$
\end{lemma}

\begin{proof}
This is very close in spirit to the main result of \cite{bonahon},
with the difference that here we consider a compact domain bounded on one
side by a pleated surface, on the other by a polyhedral surface. The argument
of \cite{bonahon} can be followed line by line, keeping one surface
polyhedral (of fixed combinatorics, say) while on the other boundary component
the approximation arguments of \cite{bonahon} can be used. 
\end{proof}

\begin{cor} \label{cr:collar}
Let $V^*(K):=V(K)+(1/2)L_\mu(\lambda)$, then, in any deformation of $K$
$$ {2V^*}' =  \sum_e L_e d\theta_e + L_\mu'(\lambda)~. $$
\end{cor}

\begin{proof}
We have seen in the proof of Lemma \ref{lm:bs} that 
$L_\mu(\lambda)' = L_\mu'(\lambda) + L_\mu(\lambda')$.
So the corollary follows from Lemma \ref{lm:collar} exactly as 
Lemma \ref{lm:bs-dual} follows from Lemma \ref{lm:bs}. 
\end{proof}

It is possible to define the renormalized volume of the complement of a polyhedral
collar in a hyperbolic end, in the same way as for quasifuchsian manifolds above. 
Let $C$ be a closed polyhedral collar in the hyperbolic end $M$, and let $D$ be its complement. 
Let $K'$ be a compact geodesically convex subset of $M$ containing $C$ in its interior, and
let $K:=K'\cap D$. We define 
$$ W(K) = V(K) -\frac{1}{4} \int_{D\cap \dr K} H da~. $$
In addition $K$ defines a metric at infinity, $I^*$, according to (\ref{eq:I*}), and
the arguments explained after Lemma 3.1 
show that $K$ is uniquely determined by $I^*$, so that $W$ 
can be considered as a function of $I^*$, a metric in the conformal class at infinity
of $M$. (In general, as explained in subsection \ref{ssec:renormalized}, 
$I^*$ only defines an equidistant foliation near
infinity which might not extend all the way to $K$.)
The first-variation of $W$ with respect to $I^*$ shows (as in \cite{volume})
that $W(I^*)$ is maximal, under the constraint that $I^*$ has fixed area, if and 
only if $I^*$ has constant curvature. We then define the renormalized volume
$V_R(D)$ as the value of this maximum.

\begin{lemma} \label{lm:complement}
Under a first-order deformation of $D$, the first-order variation of its
renormalized volume is given by 
$$ V_R(D)' = - \frac{1}{4}\int_{\dr_\infty D} \left\langle \II^*_0,
{I^*}'\right\rangle da^* + \frac{1}2 \sum_e L_e \theta_e'~. $$ 
\end{lemma}

Here $L_e$ and $\theta_e$ are the length and exterior dihedral angle of 
edge $e$ of the (polyhedral) boundary of $D$.

\begin{proof}
The proof can be obtained by following the argument used in \cite{volume},
the fact that $D$ is not complete and has a polyhedral boundary just 
adds some terms relative to this polyhedral boundary in the variations
formulae. 
\end{proof}

\begin{proof}[Proof of Proposition \ref{pr:relative}]
The statement follows directly from Corollary \ref{cr:collar} applied to 
a polyhedral collar and from Lemma \ref{lm:complement} 
applied to its complement, since the terms 
corresponding to the polyhedral boundary between the two cancel.
\end{proof}

\subsection{Proof of Theorem \ref{tm:main}}
\label{ssec:proof}

Since hyperbolic ends are in one-to-one correspondence with $\C P^1$-structures,
we can consider the relative volume $V_R$ as a function on $\CP$. 
Let $\beta_H$ (resp. $\beta_C$) be the Liouville form on $T^*\cT_H$ (resp. $T^*\cT_C$). 
We can consider the composition $\delta\circ Gr^{-1}:\CP\rightarrow T^*\cT_H$, it 
is $C^1$ and it pulls back $\beta_H$ as
$$ (\delta\circ Gr^{-1})^* \beta_H = L_\mu'(\lambda)~. $$
Under the identification of $\CP$ with $T^*\cT_C$ through the Schwarzian derivative,
the expression of $\beta_C$ is
$$ \beta_C = \int_{\dr_\infty M}\langle {I^*}',\II^*_0\rangle da^*~. $$
So Proposition \ref{pr:relative} can be formulated as
$$ dV_R = \frac{1}{2} (\delta\circ Gr^{-1})^*\beta_H - \frac{1}{4} \beta_C~, $$
and it follows that $2(\delta\circ Gr^{-1})^*\omega_H = \omega_C$.
\subsection{The Fuchsian slice vs Bers slices}

Here we prove for the reader's convenience that the identification considered here
between $\CP$ and $T^*\cT_C$, based on the Fuchsian slice, determines the same
symplectic structure on $\CP$ as the identification based on a Bers slice, as used
e.g. in \cite{kawai}. We consider a fixed conformal structure $c_-\in \cT$. Then,
for each $c\in \cT$, we call $\sigma_{c_-}(c)$ the complex projective structure
on the upper boundary at infinity of the (unique) quasifuchsian manifold for which
the lower conformal metric at infinity is $c_-$ and the upper conformal metric at
infinity is $c$. The Schwarzian derivative of the identity map from $(S,\sigma_{c_-})$ to
$(S,\sigma)$ is a holomorphic quadratic differential on $S$, and can be considered
as a point of the (complexified) cotangent space $T^*_c\cT_C$. Taking its real part
defines a map from $\CP$ to $T^*\cT_C$, which we can use to pull back the cotangent
symplectic map on $T^*\cT_C$ to a symplectic form $\omega_{c_-}$ on $\CP$.
Recall that the symplectic form $\omega_C$ considered in the paper is obtained
in the same manner, but using the Fuchsian complex projective structure $\sigma_0$
rather than the complex projective structure of the Bers slice $\sigma_{c_-}$.

\begin{lemma} \label{lm:fuchsian-bers}
$\omega_{c_-}=\omega_C$.
\end{lemma}

\begin{proof}
Consider $\sigma\in \CP$ and let $c$ be its underlying complex structure, 
let $\alpha_0(\sigma)=\cS(Id:(S,\sigma_0(c))\rightarrow (S,\sigma))$, and 
let $\alpha_{c_-}(\sigma)=\cS(Id:(S,\sigma_{c_-}(c))\rightarrow (S,\sigma))$. 
Both $\alpha_0(\sigma)$ and $\alpha_{c_-}(\sigma)$ can  be considered as vectors in the (complexified)
cotangent space $T^*_c\cT_C$. The properties of the Schwarzian derivative
under composition show that $\alpha_{0}(\sigma)-\alpha_{c_-}(\sigma)=
\cS(Id:(S,\sigma_0(c))\rightarrow (S,\sigma_{c_-}(c)))$. 
So $\alpha_{0}(\sigma)-\alpha_{c_-}(\sigma)$ depends only on the underlying
complex structure $c$ of $\sigma$ (and on $c_-$), and it defines a section of the complexified
cotangent bundle $T^*\cT_C$. By definition this is precisely the section called 
$\theta_{c_-}$ in \cite{volume} (after Theorem 8.8). 

Still by construction, $\omega_{c_-}-\omega_C=Re(d\alpha_{c_-}-d\alpha_0)=Re(d\theta_{c_-})$.
According to Proposition 8.9 of \cite{volume}, $d\theta_{c_-}$ does not depend on 
$c_-$. So $d\theta_{c_-}$ can be computed by choosing $c_-=c$ (fixed). An 
explicit computation is possible, see Proposition 8.10 in  \cite{volume},
it shows that for any two tangent vectors $X,Y\in T_c\cT_C$,
$$ (D_X Re(\theta_{c_-}))(Y)=\langle X,Y\rangle_{WP}~, $$
and it follows that $d(Re(\theta_{c_-}))=0$. So $\omega_{c_-}=\omega_0$
as claimed.
\end{proof}

\bibliographystyle{alpha}

\def\cprime{$'$}

\end{document}